\documentclass[11pt]{amsart}

\usepackage[all,color]{xy}
\usepackage{pb-diagram}
\usepackage[mathscr]{eucal}
\usepackage{hyperref}
\hypersetup{
    colorlinks=true, 
    linktoc=all,     
    linkcolor=blue,  
}

\usepackage{xcolor}

\usepackage[toc,page]{appendix}
\usepackage{pifont}
\usepackage{combelow} 

\RequirePackage{ifthen,setspace,enumitem,tikz}
\usetikzlibrary{arrows}

\DeclareMathAlphabet{\mathpzc}{OT1}{pzc}{m}{it}

\usepackage{amsfonts}
\usepackage{amsmath}
\usepackage{amssymb}
\usepackage{needspace}

\newcommand{\tr}{\textnormal{tr}}

\newcommand{\dbar}{\overline{\partial}}

\newcommand{\Ric}[2]{R_{#1 \overline{#2}}}

\newcommand{\ddbar}{\sqrt{-1}\partial\dbar}

\def\Z{{\mathbb Z}}

\newtheorem{theorem}{Theorem}[section]
\newtheorem{proposition}{Proposition}[section]
\newtheorem{lemma}{Lemma}[section]

\numberwithin{equation}{section}

\renewcommand{\Ric}{\operatorname{Ric}}
\newcommand{\dc}{d^c}
\newcommand{\eps}{\varepsilon}
\newcommand{\Nash}{\mathcal N}
\newcommand{\HH}{\mathcal H}
\newcommand{\RR}{\mathcal R}
\renewcommand{\Z}{\mathcal Z}

\newcommand{\ip}[2]{\left\langle #1,#2\right\rangle}
\setlist[enumerate]{leftmargin=*,label=\textup{(\roman*)},itemsep=3pt}
\allowdisplaybreaks[1]
\hypersetup{pdftitle={Type I estimates for the Kahler-Ricci flow I},
 pdfauthor={Wangjian Jian and Jian Song},
 pdfsubject={Type I scalar curvature estimates at finite-time singularities}}
\title[Type I estimates for the K\"ahler-Ricci flow]{Type I estimates for the K\"ahler-Ricci flow I}
\author{Wangjian Jian}
\address{Institute of Mathematics, Academy of Mathematics and Systems Science,
Chinese Academy of Sciences, Beijing, 100190, China}
\email{wangjian@amss.ac.cn}
\thanks{Wangjian Jian is supported in part by NSFC grants 12422103, 12201610,
12371058, and 12288201; BJNSF grant JR25002; and National Key R\&D Program
of China grants 2023YFA1009900 and 2021YFA1003100.}
\author{Jian Song}
\address{Department of Mathematics, Rutgers University,
Piscataway, NJ 08854, USA}
\email{jiansong@math.rutgers.edu}
\thanks{Jian Song is supported in part by the National Science Foundation
grant DMS-2505575.}

\date{}
\begin{document}
\raggedbottom
\begin{abstract}
We establish the Type I estimate for the scalar curvature of the
K\"ahler-Ricci flow on compact K\"ahler manifolds, for arbitrary smooth
initial metrics, whenever the flow develops finite time singularities. This estimate is a building block for the analytic
minimal model program with Ricci flow.
\end{abstract}
\maketitle

\section{Introduction}\label{sec:introduction}

The Ricci flow was introduced by Hamilton \cite{Hamilton}. In the
K\"ahler setting, it offers a natural approach to the classification of
compact K\"ahler manifolds, through the evolution of their metrics.
This is the viewpoint of the analytic minimal model program proposed by
Song and Tian \cite{STMeasures,ST}, which seeks to realize the
contractions, flips and fibre collapsing of birational geometry through
the formation of singularities of the flow and their continuation. When
the canonical bundle is semi-ample, their work on canonical measures
describes the convergence of the flow towards canonical metrics
\cite{STMeasures}, while in the projective setting, their theory of the
weak K\"ahler-Ricci flow provides a continuation through divisorial
contractions and flips \cite{ST}. Song and Weinkove \cite{SW}
established Gromov-Hausdorff continuation across certain contractions
of exceptional divisors, giving concrete instances of the geometric
surgeries predicted by the program. A central problem which remains is
to understand the metric geometry of the flow at a general finite-time
singularity.

A fundamental model for this study is Perelman's work on the
K\"ahler-Ricci flow on Fano manifolds. His entropy and noncollapsing
theorems \cite{Perelman}, together with his estimates for the normalized
flow in the first Chern class, give uniform bounds for the scalar
curvature, the diameter and the Ricci potential (c.f. \cite{SesumTian}). For the
unnormalized flow, these become Type I bounds at the extinction time.
An important lesson of Perelman's estimates is that curvature and
distance bounds can be derived from a Ricci potential, even when the
full curvature tensor is not under control. A basic objective of the
analytic minimal model program is to extend this picture to flows which
contract subvarieties or collapse onto spaces of positive dimension.

Beyond the Fano case, scalar curvature estimates have also played an
essential role. When the canonical bundle is semi-ample, Song and Tian
\cite{STScalar} proved a uniform bound for the scalar curvature along
the normalized flow. For finite-time singularities, Zhang
\cite{ZhangFinite} showed that the scalar curvature must become
unbounded, and obtained an upper bound under an additional assumption on
the limiting class. The question addressed in the present paper is the
sharp rate at which the scalar curvature can grow at a general
finite-time singularity.

Our approach rests above all on the work of Jian, Song and Tian
\cite{JST}, who proposed a general Type I conjecture for the scalar
curvature and the fibre diameters, extending Perelman's estimates to
finite-time contractions and collapsing. They introduced the notion of
Ricci vertices and developed Li-Yau type estimates for weighted Ricci
potentials, together with Harnack and distance estimates. In particular,
they obtained Type I control of the scalar curvature near Ricci
vertices, as well as Gromov-Hausdorff compactness for Type I blow-ups
based at suitable points. Their complex geometric compactness theory,
which includes local partial $C^0$-estimates, produces limits on
analytic normal varieties. They also established fibre diameter
estimates for collapsing Fano bundles, and proved Type I singularity
formation for flows with Calabi symmetry \cite{JST}. Taken together,
these results tie the analysis of Ricci potentials to the geometric and
complex analytic structure of finite-time singularities.

In complex dimension two, Conlon, Hallgren and Ma \cite{CHM} proved the
Type I bound for the full curvature tensor at noncollapsed finite-time
singularities, and the volume-collapsing case was subsequently treated by Cifarelli, Conlon, Hallgren and Zhang \cite{CCHZ} and by Xu and Zhang \cite{XZ}, completing the
Type I picture for compact K\"ahler surfaces. For collapsing Fano
bundles, Jian and Song \cite{JSFanoII} established Type I estimates for
the scalar curvature and the fibre diameters for arbitrary smooth Fano
fibres, as well as a Type I bound for the full curvature tensor when the
fibre admits a smooth K\"ahler-Einstein metric.

\bigskip

Let $X$ be a compact   K\"ahler manifold of complex dimension $n$,
and let $g_0$ be a smooth K\"ahler metric on $X$. We consider the maximal
unnormalized K\"ahler-Ricci flow
\begin{equation}\label{eq:flow}
 \frac{\partial\omega}{\partial t}=-\Ric(\omega),\qquad
 \omega(0)=\omega_0,\qquad 0\le t<T<\infty,
\end{equation}
where $\omega_0$ is the K\"ahler form associated with $g_0$. We denote by $g(t)$ the
corresponding metric associated to $\omega(t)$, and by $R=R_{g(t)}=\tr_\omega\Ric(\omega)$ its scalar
curvature, normalized as in Section~\ref{sec:basic}.

The main result of this paper is the following theorem.

\begin{theorem}\label{thm:main}
Let $X$ be a compact K\"ahler manifold and let $g(t)$ be the maximal
solution of the K\"ahler-Ricci flow \eqref{eq:flow} with smooth initial
K\"ahler metric $g_0$. Assume that the maximal existence time
$T>0$ is finite. Then there exists a constant $C>0$, depending on $X$ and
$g_0$, such that
\begin{equation}\label{eq:main}
 \sup_X |R_{g(t)}|\le\frac{C}{T-t},\qquad 0\le t<T.
\end{equation}
\end{theorem}

Theorem~\ref{thm:main} establishes the scalar curvature part of the
conjecture of Jian, Song and Tian \cite{JST}, for arbitrary smooth
initial K\"ahler metrics on compact K\"ahler manifolds. No restriction
is placed on the initial K\"ahler class $[\omega_0]$, and neither
rationality nor projectivity is assumed. It is the transcendental
base-point-free theorem of Hacon and Xie \cite{HX} which allows us to
treat the limiting cohomology class in this generality. We view
Theorem~\ref{thm:main} as a building block of the analytic minimal model
program with Ricci flow. Indeed, it provides uniform control of the
scalar curvature at the natural scale of the singularity, and such
control is needed in order to relate the degeneration of the metrics to
the contraction determined by the limiting class.

Our starting point is the global Ricci potential $U$ introduced by Jian
and Song in \emph{Finite-Time Singularities of the K\"ahler-Ricci Flow on
Fano Bundles II} \cite{JSFanoII}, which builds on the Ricci-potential
framework of Jian, Song and Tian \cite{JST}. This global potential is
central to our argument, since it allows all regions of the flow to be
studied within a single global construction. The main new idea is to
consider the one-form $\dc U$ modulo $d$-closed one-forms, a formulation
which isolates the part of the potential that is relevant to curvature.

The main advantage of this viewpoint is that it eliminates the gauge
freedom in the potential. Taking the differential removes additive
constants, while modifying local potentials by pluriharmonic functions
only adds closed one-forms to $\dc U$, and these contributions are
removed as well once we pass to the quotient by closed one-forms. Thus
the quantity we estimate carries no energy arising solely from such
local choices of potential. Proposition~\ref{prop:scalar} makes this
idea effective: it controls the scalar curvature by the distance from
$\dc U$ to the closed one-forms, measured with respect to the conjugate
heat measure. This is the key technical and conceptual ingredient of the
proof.

To control this distance near a singularity, we rely on Bamler's theory
of entropy, heat kernels and compactness
\cite{BamHeat,BamCompact,BamStructure}, together with the work of
Hallgren and Jian \cite{HJ} and of Hallgren \cite{Hallgren} on
K\"ahler-Ricci tangent flows. These results allow us to pass to limiting shrinkers
and to study the one-forms there. A vanishing argument
on these limits then yields the improvement needed for the Type I
estimate.

In a forthcoming preprint, \emph{Type I estimates for the K\"ahler-Ricci
flow II} \cite{JSTypeII}, we shall develop applications of
Theorem~\ref{thm:main} to the geometry of finite-time singularities.
These include global and local diameter estimates, the identification of
the Gromov-Hausdorff limit space of the unnormalized flow with the
K\"ahler space induced by the limiting cohomology class, and global and
local partial $C^0$-estimates. We shall also show that Type I blow-up
limits are K\"ahler-Ricci soliton spaces carrying the structure of
normal complex analytic spaces. These applications will further connect
the scalar curvature estimate with the geometric realization of the
analytic minimal model program.

The paper is organized as follows. In Section~\ref{sec:basic}, we
collect the basic estimates for the limiting class, the heat kernel and
one-forms evolving by the Hodge heat flow. Section~\ref{sec:limits}
recalls the compactness and almost-soliton results used in the proof.
The Ricci potential is introduced in Section~\ref{sec:potential}, and
the integral estimate for the scalar curvature is established in
Section~\ref{sec:scalar}. The vanishing argument is given in
Section~\ref{sec:vanishing}, and the contracting iteration which
completes the proof is carried out in Section~\ref{sec:proof}.

\bigskip
\noindent\textbf{Acknowledgements.} We would like to thank Professor Gang Tian for his generous support and encouragement over the years, which have been a constant source of inspiration. This work was carried out with the assistance of ChatGPT. Our original
idea was to bound the scalar curvature by the $L^2$-energy of the
gradient of the global potential $U$ with respect to the conjugate heat
measure, but this approach turned out to be unsuccessful. ChatGPT made
the surprising correction of eliminating the $d$-closed kernel, and
provided Proposition~\ref{prop:scalar}, which is the key conceptual and technical ingredient of this paper.

\bigskip

\section{Basic estimates}\label{sec:basic}

Replacing $\omega(t)$ by $T^{-1}\omega(Tt)$, we may assume that $T=1$,
and we keep the same notation for the rescaled solution and its initial
metric. For the real metric, we use the normalization
\begin{equation}\label{eq:conventions}
 g=2\omega(\,\cdot\,,J\,\cdot\,),\qquad
 \partial_tg=-2\Ric_g,\qquad
 \Delta_g=\Delta=\omega^{i\bar j}\partial_i\partial_{\bar j}.
\end{equation}
Thus $g(0)=g_0$ before the time normalization, and
$R_g=\tr_\omega\Ric(\omega)=R$. The constant factor in
\eqref{eq:conventions} has no effect on the Type I estimate.
Throughout the paper, $C$ denotes a positive constant depending only on
the fixed initial data and on auxiliary background choices, whose value
may change from line to line. We also write
\[
 \square=\partial_t-\Delta,\qquad
 \dc=\sqrt{-1}(\bar\partial-\partial),\qquad
 d\dc=2\ddbar.
\]
Real norms, adjoints and covariant derivatives are taken with respect to
$g$, while norms of complex tensors are taken with respect to $\omega$.
In particular,
\begin{equation}\label{eq:norms}
 |du|_g^2=|\partial u|_\omega^2,\qquad
 |\beta|_g^2=|\beta^{0,1}|_\omega^2
 \quad\hbox{for every real one-form }\beta.
\end{equation}

\subsection{Preliminary estimates}\label{subsec:class}

By the maximal existence theorem of Tian and Zhang
\cite{TianZhang}, which builds on earlier work of Tsuji \cite{Tsuji},
the maximal existence time is characterized by
\[
 T=\sup\{t>0:[\omega_0]-2\pi t c_1(X)\text{ is a K\"ahler class}\}.
\]
With the normalization $T=1$, the limiting class
\[
 \alpha=[\omega_0]-2\pi c_1(X)
       =[\omega_0]+2\pi c_1(K_X)
\]
is nef. We shall use the following smooth adjoint case of the
transcendental base-point-free theorem.

\begin{theorem}[Hacon-Xie {\cite[Theorem 1.4]{HX}}]\label{thm:HX}
Let $X$ be a compact K\"ahler manifold and let
$\beta\in H^{1,1}(X,\mathbb R)$ be a K\"ahler class.
If $c_1(K_X)+\beta$ is nef, then there exist a surjective holomorphic
map $\Phi:X\to Y$ with   fibres onto a compact normal
K\"ahler space and a K\"ahler class $\gamma$ on $Y$ such that
\[
 c_1(K_X)+\beta=\Phi^*\gamma
 \quad\hbox{in }H^{1,1}(X,\mathbb R).
\]
\end{theorem}

Applying Theorem~\ref{thm:HX} with $\beta=[\omega_0]/(2\pi)$ and
multiplying the resulting class on $Y$ by $2\pi$, we obtain a K\"ahler
form $\theta_Y$ such that
\begin{equation}\label{eq:terminal}
 \theta=\Phi^*\theta_Y\ge0,\qquad [\theta]=\alpha.
\end{equation}
Here a smooth K\"ahler form on $Y$ is understood in the usual sense,
namely as given by local smooth strictly plurisubharmonic potentials in
ambient embeddings; its pullback is a smooth closed semipositive form on
$X$. If $Y$ is a point, we set $\theta=0$.

We choose a smooth positive volume form $\Omega$ on $X$ satisfying
\[
 \Ric(\Omega)=\omega_0-\theta.
\]
This is possible by the $\partial\bar\partial$-lemma. Setting
$\omega_t=(1-t)\omega_0+t\theta$, the flow
\eqref{eq:flow} is equivalent to the complex Monge-Amp\`ere flow
\begin{equation}\label{eq:MA}
 \left\{
 \begin{aligned}
 \frac{\partial\varphi}{\partial t}
   &=\log\frac{(\omega_t+\ddbar\varphi)^n}{\Omega},\\
 \varphi(\cdot,0)&=0,
 \end{aligned}\right.
 \qquad
 \omega(t)=\omega_t+\ddbar\varphi>0.
\end{equation}

\begin{proposition}\label{prop:basic}
There exist constants $c,C>0$ such that, on $X\times[0,1)$,
\begin{equation}\label{eq:C0Schwarz}
 \|\varphi(t)\|_{L^\infty(X)}\le C,\qquad
 0\le \tr_\omega(\theta)\le C,
\end{equation}
and
\begin{equation}\label{eq:basictrace}
 \square \tr_\omega(\theta)\le-c|\partial \tr_\omega(\theta)|_\omega^2+C,
 \qquad R\ge-C.
\end{equation}
\end{proposition}

\begin{proof}
The background forms satisfy
$(1-t)\omega_0\le\omega_t\le C\omega_0$. Hence \eqref{eq:MA} gives
$\partial_t\varphi\le C$ at a spatial maximum of $\varphi$, and
$\partial_t\varphi\ge n\log(1-t)-C$ at a spatial minimum.
The maximum principle therefore yields
\[
 \int_0^t\bigl(n\log(1-s)-C\bigr)\,ds
 \le\varphi(x,t)\le Ct.
\]
Since $\log(1-s)$ is integrable on $[0,1)$, the first estimate follows.
The evolution equation $\square R=|\Ric(\omega)|_\omega^2$ for the scalar
curvature gives $R\ge-C$.

To prove the Schwarz estimate, we set $H=(1-t)\dot\varphi+\varphi$.
Differentiating \eqref{eq:MA} gives
\[
 H_t=-(1-t)R,\qquad \square H=\tr_\omega(\theta)-n.
\]
The lower bound for the scalar curvature gives $H\le C$, while
$\tr_\omega(\theta)\ge0$ and the minimum principle give $H\ge-C$.

Consider a local embedding of $Y$ into a complex Euclidean space, and
extend a local potential for $\theta_Y$ to a smooth strictly
plurisubharmonic function. The resulting ambient K\"ahler metric has
bounded curvature on a smaller chart, and the parabolic Schwarz
calculation for the holomorphic map obtained by composing $\Phi$ with
this embedding yields
\begin{equation}\label{eq:Schwarzlocal}
 \begin{aligned}
 \square\log\tr_\omega(\theta)&\le C\tr_\omega(\theta)
       &&\text{where }\tr_\omega(\theta)>0,\\
 \square\tr_\omega(\theta)&\le-c_0|\nabla d\Phi|^2
       +C\bigl(\tr_\omega(\theta)\bigr)^2.
 \end{aligned}
\end{equation}
Here the second covariant derivative is computed with respect to the
chosen ambient metric. By the compactness of $Y$, finitely many such
charts suffice, with uniform constants. Thus the first inequality holds
as a global pointwise inequality wherever $\tr_\omega(\theta)>0$.

For a fixed sufficiently large constant $A$, we have
\[
 \square(\log \tr_\omega(\theta)-AH)\le-(A-C)\tr_\omega(\theta)+An.
\]
The maximum principle and the bound $|H|\le C$ then imply
$\tr_\omega(\theta)\le C$. Here the argument is applied at positive
maxima of $\tr_\omega(\theta)$, or else to a positive regularization;
when $\Phi$ is constant, the conclusion is immediate.
On each ambient chart, we also have
$|\partial \tr_\omega(\theta)|^2\le C\tr_\omega(\theta)|\nabla d\Phi|^2$.
Together with the bound for $\tr_\omega(\theta)$, the second inequality
in \eqref{eq:Schwarzlocal} now gives the first inequality in
\eqref{eq:basictrace}. These are the local Schwarz calculations used in
\cite{JST}, and they require only the local ambient description of
$\theta_Y$.
\end{proof}

\subsection{Nash entropy and conjugate heat kernel}\label{subsec:heat}

The conjugate heat kernel based at $p=(x,t)$ is the positive function
$K(x,t;y,s)$, defined for $s<t$, which satisfies
\[
 \bigl(-\partial_s-\Delta_{g(s)}+R(\cdot,s)\bigr)
 K(x,t;\cdot,s)=0.
\]
The associated measures
\[
 d\nu_{x,t;s}(y)=K(x,t;y,s)\,dV_{g(s)}(y)
\]
are probability measures, and they converge weakly to $\delta_x$ as
$s\nearrow t$.
Since $\partial_s dV_{g(s)}=-R\,dV_{g(s)}$, an integration by parts
gives the fundamental duality identity
\begin{equation}\label{eq:duality}
 \frac{d}{ds}\int_X F(\cdot,s)\,d\nu_{x,t;s}
 =\int_X\square F(\cdot,s)\,d\nu_{x,t;s}
\end{equation}
for every smooth spacetime function $F$.

For $0\le s<t$, we write $\tau=t-s$ and define the heat potential by
\[
 d\nu_{x,t;s}=(4\pi\tau)^{-n}e^{-f_{p,\tau}}dV_{g(s)}.
\]
The pointed Nash entropy is then defined by
\begin{equation}\label{eq:Nashdef}
 \Nash_p(\tau)=\int_X f_{p,\tau}\,d\nu_{x,t;s}-n.
\end{equation}
Recall that the real dimension is $2n$; with this convention,
Perelman's pointed $\mathcal W$-entropy is
\[
 \mathcal W_p(\tau)=\int_X
 \bigl(\tau(|df_{p,\tau}|^2+R)+f_{p,\tau}-2n\bigr)\,d\nu_{x,t;s}.
\]
In these formulas, the metric and the curvature are evaluated at
$s=t-\tau$. For a fixed base point $p$, Perelman's entropy identities
\cite{Perelman} read
\begin{align}
 \frac{d}{d\tau}\bigl(\tau\Nash_p(\tau)\bigr)&=\mathcal W_p(\tau),
       \label{eq:NashW}\\
 \frac{d}{d\tau}\mathcal W_p(\tau)
 &=-2\tau\int_X\left|\Ric_g+\nabla^2 f_{p,\tau}
                   -\frac{g}{2\tau}\right|^2d\nu_{x,t;s}.
       \label{eq:Wderivative}
\end{align}

\begin{lemma}\label{lem:entropy}
There exists a constant $Y<\infty$ such that
\begin{equation}\label{eq:entropy}
 -Y\le\Nash_p(\tau)\le0,\qquad 0<\tau\le t<1.
\end{equation}
Moreover, for each fixed base point $p$, the function
$\tau\mapsto\Nash_p(\tau)$ is nonincreasing.
\end{lemma}

\begin{proof}
By the small-time asymptotics, we have
$\mathcal W_p(\tau)\to0$ and $\Nash_p(\tau)\to0$ as $\tau\searrow0$.
Equations \eqref{eq:NashW}-\eqref{eq:Wderivative} then imply
\[
 \Nash_p(\tau)=\frac1\tau\int_0^\tau\mathcal W_p(a)\,da,
 \qquad \mathcal W_p(\tau)\le\Nash_p(\tau)\le0,
\]
which proves the monotonicity. If $\mu(g,\tau)$ denotes the infimum of
Perelman's $\mathcal W$-functional over probability densities, then
\[
 \mathcal W_p(\tau)\ge\mu(g(t-\tau),\tau)\ge\mu(g(0),t).
\]
Since the initial metric is smooth and fixed, it has a uniform lower
entropy bound for $0<t\le1$, and averaging in $\tau$ proves
\eqref{eq:entropy}.
\end{proof}

Bamler's work \cite{BamHeat} provides heat-kernel and entropy estimates
whose constants are controlled by the dimension, a lower bound for the
scalar curvature, and the pointed entropy. His compactness theory
\cite{BamCompact} compares metric flows by means of their conjugate heat
measures, while his structure theory \cite{BamStructure} describes the
regular and singular parts of noncollapsed limits. In particular, a
small drop of the entropy yields shrinking limits, and one can pass to
smooth limits on their regular spacetime. The precise form needed here
is stated in Section~\ref{sec:limits}.

We next introduce the scalar heat evolution used in the estimates below.
For times $a<b$ in the interval of a smooth flow, let $\Phi_{a,b}$
denote its heat solution operator. Thus, for $F\in C^\infty(X)$, the
function $v(\cdot,s)=\Phi_{a,s}F$ is the unique smooth solution of
\[
 \square v=0\quad\hbox{on }X\times[a,b],\qquad v(\cdot,a)=F.
\]
It admits the representation
\[
 (\Phi_{a,b}F)(x)=\int_X F(y)\,d\nu_{x,b;a}(y).
\]
These operators preserve constants and positivity, and satisfy
$\Phi_{b,c}\Phi_{a,b}=\Phi_{a,c}$. If $\square v=G$, then Duhamel's
principle gives
\begin{equation}\label{eq:DuhamelScalar}
 v(b)=\Phi_{a,b}v(a)+\int_a^b\Phi_{s,b}G(s)\,ds.
\end{equation}
We shall use the same notation after time has been shifted and rescaled.

\begin{lemma}\label{lem:heat}
For $F\in C^\infty(X)$ and $a<b$ in the interval of the flow,
\begin{equation}\label{eq:heatgradient}
 |d\Phi_{a,b}F|\le \frac{C}{\sqrt{b-a}}\|F\|_\infty.
\end{equation}
For a fixed base point at time $0$, write $\nu_s=\nu_{x,0;s}$.
Then every nonnegative heat subsolution satisfies
\begin{equation}\label{eq:hyper}
 \|F(b)\|_{L^q(\nu_b)}\le\|F(a)\|_{L^p(\nu_a)}
 \quad\text{if}\quad
 a<b<0,\quad 1<p\le q,\quad\frac{|a|}{|b|}\ge\frac{q-1}{p-1}.
\end{equation}
In particular, the evolution from time $-8$ to time $-1$ maps $L^2$
into $L^8$.
\end{lemma}

\begin{proof}
For $v=\Phi_{a,s}F$, the Bochner formulas for the real Ricci flow read
\[
 \square v^2=-2|dv|^2,\qquad
 \square|dv|^2=-2|\nabla dv|^2.
\]
Thus $(s-a)|dv|^2+v^2/2$ is a subsolution, and the maximum principle
proves \eqref{eq:heatgradient}. Estimate \eqref{eq:hyper} is Bamler's
heat hypercontractivity theorem \cite[Theorem 12.1]{BamHeat}, applied by
comparison when $F$ is a subsolution; it requires no upper bound on the
curvature. Finally, taking $a=-8$, $b=-1$, $p=2$ and $q=8$ is
admissible since $|a|/|b|=8\ge(q-1)/(p-1)=7$.
\end{proof}

\subsection{Hodge heat flow}\label{subsec:forms}

Let $d$ be the exterior derivative, and let $\delta=\delta_{g(t)}$ be
its formal adjoint with respect to $dV_{g(t)}$. The nonnegative Hodge
Laplacian on differential forms is
\[
 \Delta_d=d\delta+\delta d.
\]
On functions, $\Delta_d=-\Delta_g$. Let $\nabla$ be the Levi-Civita
connection of $g$, and let $\nabla^*$ be its formal adjoint with respect
to $dV_g$ at each fixed time. For a one-form $\beta$, we define
\[
 (\Delta^{\mathrm{rough}}\beta)_j
 =g^{ab}\nabla_a\nabla_b\beta_j,\qquad
 \Delta^{\mathrm{rough}}=-\nabla^*\nabla,
\]
where the indices refer to real coordinates.

We consider the time-dependent Hodge heat equation
\begin{equation}\label{eq:HodgeEquation}
 \partial_t\beta=-\Delta_{d,g(t)}\beta.
\end{equation}
For smooth initial data, this linear parabolic equation has a unique
smooth solution on every compact time interval contained in $[0,1)$,
and we denote its propagator by $\HH_{a,b}$.

We first recall a weighted version of the Hodge decomposition. At a
fixed time, let $d\nu=ce^{-f}dV_g$ be a smooth positive probability
measure. The weighted adjoint of $d$ is then
\[
 \delta_f=e^f\delta e^{-f}=\delta+\iota_{\nabla f}.
\]
We let $\Z$ denote the $L^2(\nu,g)$ closure of the space of smooth
closed real one-forms.

\begin{lemma}\label{lem:projection}
A smooth real one-form $\alpha$ admits a unique orthogonal decomposition
\begin{equation}\label{eq:projection}
 \alpha=\zeta+\beta,\qquad \zeta\in\Z,\quad\beta\perp\Z.
\end{equation}
Both $\zeta$ and $\beta$ are smooth, and they satisfy  $$d\zeta=0, ~\delta_f\beta=0. $$
Equivalently, $\beta$ minimizes $\int|\widetilde\beta|^2d\nu$ among
real forms $\widetilde\beta$ satisfying $d\widetilde\beta=d\alpha$.
\end{lemma}

\begin{proof}
The weighted Hodge Laplacian $L_f=d\delta_f+\delta_fd$ is elliptic and
self-adjoint on the compact manifold. Let $\Pi$ denote the projection
onto its kernel, and $\mathcal G$ its Green operator. The usual Hodge
decomposition then reads
\[
 \zeta=\Pi\alpha+d\delta_f\mathcal G\alpha,
 \qquad \beta=\delta_fd\mathcal G\alpha.
\]
Smoothness follows from elliptic regularity. Since the quadratic
form of $L_f$ is $\|d\cdot\|_2^2+\|\delta_f\cdot\|_2^2$, the form
$\Pi\alpha$ is closed. Moreover, $\delta_f^2=0$, and for every smooth
closed $\gamma$,
\[
 \int\ip{\beta}{\gamma}\,d\nu
 =\int\ip{d\mathcal G\alpha}{d\gamma}\,d\nu=0.
\]
This proves \eqref{eq:projection}. The minimizing characterization
follows, since $d\widetilde\beta=d\alpha$ if and only if
$\alpha-\widetilde\beta$ is closed.
\end{proof}

\begin{lemma}\label{lem:hodgeheat}
The propagator $\HH_{a,b}$ of the Hodge heat equation preserves real
forms and closed forms, and satisfies
\begin{align}
 &\HH_{a,b}(\dc v)=\dc \Phi_{a,b}v.\label{eq:commute}
\end{align}
 Moreover, if we let $\beta$ be the solution of  $$\partial_t\beta=-\Delta_d\beta. $$
Then
 \begin{align}
 &\square|\beta|^2=-2|\nabla \beta|^2,\label{eq:hodgebochner}\\
 &\int|\beta(t)|^2d\nu_t+2\int_a^t\!\int|\nabla \beta|^2d\nu_s\,ds
 =\int|\beta(a)|^2d\nu_a.\label{eq:hodgeenergy}
\end{align}
Similarly, for a complex $(0,1)$-form $\eta$ solving
$\partial_t\eta=-\Delta_{\bar\partial,\omega(t)}\eta$, where
$\Delta_{\bar\partial,\omega}$ is the nonnegative Dolbeault Laplacian
of the Hermitian metric $\omega$, we have
\begin{equation}\label{eq:complexheat}
 \square|\eta|^2=-|\nabla^{1,0}\eta|^2-|\nabla^{0,1}\eta|^2.
\end{equation}
\end{lemma}

\begin{proof}
The equation has real coefficients, so it preserves real forms. As
$d\Delta_d=\Delta_dd$ at each time and $d$ is independent of time,
$d\beta$ solves the Hodge heat equation in degree two, and uniqueness
shows that the condition $d\beta=0$ is preserved. The K\"ahler
identities and the fact that the complex structure is fixed give
\eqref{eq:commute}; indeed, under \eqref{eq:conventions},
$\Delta_{d,g}=\Delta_{\bar\partial,\omega}$ on complexified forms.

By the Weitzenb\"ock formula, the equation can be written as
$\partial_t\beta=\Delta^{\rm rough}\beta-\Ric_g^{\#}\beta$. In
$\partial_t|\beta|^2$, the variation of the inverse metric contributes
$2\Ric_g(\beta,\beta)$, which exactly cancels the zero-order term coming
from the equation. This proves \eqref{eq:hodgebochner}, and
\eqref{eq:hodgeenergy} then follows from \eqref{eq:duality}. The same
calculation for complex $(0,1)$-forms gives \eqref{eq:complexheat};
compare \cite[Lemma 3.1(i)]{Hallgren}. Finally, applying the chain rule
to $(|\beta|^2+\eps^2)^{1/2}$ shows that $|\beta|$ is a nonnegative heat
subsolution.
\end{proof}

\section{A compactness theorem of Bamler}\label{sec:limits}

We shall need the following consequence of Bamler's compactness and
structure theory \cite{BamCompact,BamStructure}, which we formulate so
as to include the strong almost-soliton potentials of Hallgren and Jian
\cite{HJ,Hallgren}, as well as the cutoff properties required for the
vanishing argument in Section~\ref{sec:vanishing}.

\begin{theorem}\label{thm:compactness}
Let $p_i=(x_i,t_i)$ be points of the fixed flow with $t_i\to1$, and let
$\tau_i\to0$ be positive time scales. Suppose that there exist
$K_i\to\infty$ such that $K_i\tau_i\le t_i$ and
\begin{equation}\label{eq:smallentropysequence}
 \Nash_{p_i}(\tau_i/K_i)-\Nash_{p_i}(K_i\tau_i)\longrightarrow0.
\end{equation}
Consider the rescaled flows and their conjugate heat measures
\[
 g_i(s)=\tau_i^{-1}g(t_i+\tau_is),\qquad
 d\nu_{i,s}=(4\pi|s|)^{-n}e^{-f_i}dV_{g_i(s)},\qquad s<0.
\]
Then, after passing to a subsequence, these flows converge to a possibly
singular ancient shrinking K\"ahler-Ricci flow with regular spacetime
$\RR$ and conjugate heat measures $\nu_s$, and the following properties
hold.
\begin{enumerate}
 \item The metrics, the complex structures, the time vector fields and
 the functions $f_i$ converge smoothly on compact subsets of $\RR$. For
 every $s<0$, $\nu_s(\RR_s)=1$, and the limiting heat potential
 satisfies
 \[
 \Ric_g+\nabla^2f=\frac{g}{2|s|}\quad\hbox{on }\RR_s.
 \]
 \item For every compact interval $I\subset(-\infty,0)$, there exist
 smooth functions $h_i$ on $X\times I$ and constants $c_i$ such that,
 with $\tau=-s$ in these rescaled coordinates,
 \begin{align}
  &\tau\square h_i=h_i-c_i,\label{eq:strongheat}\\
  &\sup_{s\in I}\int_X|d(f_i-h_i)|^2d\nu_{i,s}\longrightarrow0,
       \label{eq:stronggradient}\\
  &\int_I\!\int_X\left|\Ric_{g_i}+\nabla^2h_i-\frac{g_i}{2\tau}\right|^2
       d\nu_{i,s}\,ds\longrightarrow0.\label{eq:strongerror}
 \end{align}
 \item On the $(-1)$-time slice, bounded balls have finite regular
 volume, and the heat density is locally bounded above. On each bounded
 ball $B$, there exist smooth cutoff functions $0\le\chi_\eps\le1$ on
 the regular part, vanishing near the singular set, such that
 \[
 \chi_\eps\longrightarrow1\quad\hbox{locally on }\RR_{-1},\qquad
 \|d\chi_\eps\|_{L^3(B\cap\RR_{-1},dV_g)}\longrightarrow0.
 \]
 Moreover, fixing a point $z_0\in\RR_{-1}$, there exist smooth cutoff
 functions $0\le\rho_R\le1$ on the regular part, indexed by the radius
 $R\ge1$ of a ball centered at $z_0$ in the limit metric on the
 $(-1)$-time slice, such that
 \[
 \rho_R=1\ \hbox{on }B(z_0,R/2)\cap\RR_{-1},\qquad
 \operatorname{supp}\rho_R\subset B(z_0,R)\cap\RR_{-1},\qquad
 |d\rho_R|\le C/R.
 \]
 Thus $\rho_R\to1$ locally as $R\to\infty$, and
 $\rho_R\chi_\eps$ has compact support in $\RR_{-1}$.
\end{enumerate}
\end{theorem}

\begin{proof}
The entropy lower bound \eqref{eq:entropy} and the lower bound
$R_{g_i}\ge-C\tau_i$ for the rescaled scalar curvature allow us to apply
Bamler's compactness and structure theory
\cite{BamCompact,BamStructure}, and the backward lifetime of the
rescaled sequence tends to infinity. By
\cite[Proposition 7.1]{BamStructure}, a small entropy drop implies
almost selfsimilarity, and taking a diagonal sequence produces the
shrinking limit by \cite[Theorem 15.69]{BamStructure}. The smooth
convergence on the regular part and the full mass at fixed negative
times follow from
\cite[Section 15.1, Theorem 15.28, Corollaries 15.47 and 17.2]{BamStructure}.
Since the complex structures are parallel, they pass to a smooth limit
on this regular spacetime along a subsequence.

The strong almost-soliton potentials are provided by \cite{HJ}, in the
form stated in \cite[Definition 2.3, Proposition 2.4]{Hallgren}; we
construct them on an interval slightly larger than the one under
consideration. They satisfy \eqref{eq:strongheat}, where $c_i$ equals an
entropy constant plus $n$. The comparison of their gradients at each
time is \cite[Lemma 2.5]{Hallgren}; for the reader's convenience, we
note that it follows from an ordinary weighted integration by parts.
Define, for this paragraph only,
\[
 A(v)=\tau(R+2\Delta v-|dv|^2)+v-2n-\Nash_{p_i}(\tau_i).
\]
Then
\[
 \tau\int|d(f_i-h_i)|^2d\nu_{i,s}
 =\int\bigl(A(f_i)-A(h_i)-f_i+h_i\bigr)d\nu_{i,s}.
\]
By almost selfsimilarity and the construction of the strong potentials,
the terms on the right-hand side tend to zero uniformly on the chosen
interval. This proves \eqref{eq:stronggradient}, while the integrated
soliton error gives \eqref{eq:strongerror}.

Finally, the structure theorem for solitons
\cite[Theorem 2.18, Theorem 2.16(d), Definition 2.15(4)]{BamStructure}
gives, on a fixed slice, spatial Minkowski codimension at least four for
the singular set, together with local volume bounds. Consequently, in a
fixed ball, the volume of an $\eps$-neighborhood of the singular set is
at most $C\eps^{4-\gamma}$, for any fixed $0<\gamma<1$. Since a distance
cutoff has gradient at most $C/\eps$, we obtain
\[
 \int|d\chi_\eps|^3dV_g\le C\eps^{-3}\eps^{4-\gamma}
 =C\eps^{1-\gamma}\longrightarrow0.
\]
Smooth approximation on the regular locus then suffices. The
cutoffs $\rho_R$ are likewise obtained from distance functions, and the
upper bounds for the heat kernel give the local boundedness of the
density. This completes the proof of the stated consequence of the
compactness, structure and almost-soliton results cited above.
\end{proof}

\section{Ricci potentials}\label{sec:potential}

Let $\varphi$ be the solution of \eqref{eq:MA}. Following the
Ricci-potential framework of Jian, Song and Tian \cite{JST}, we use the
global Ricci potential of Jian and Song \cite{JSFanoII} in the
unnormalized form
\begin{equation}\label{eq:Udef}
 U=(1-t)\dot\varphi+\varphi+nt.
\end{equation}
This normalization is convenient when comparing time differences
$\tau=t-s$ which need not be comparable to $1-t$.

\begin{lemma}\label{lem:basic}
There exist constants $c,C>0$ such that
\begin{equation}\label{eq:potential}
 \begin{gathered}
 \ddbar U=\omega-\theta-(1-t)\Ric(\omega),\\
 U_t=n-(1-t)R,\qquad \square U=\tr_\omega(\theta).
 \end{gathered}
\end{equation}
\begin{equation}\label{eq:basicU}
 |U|+|\partial U|_\omega^2\le C.
\end{equation}
Moreover, $\tr_\omega(\theta)$ and $R$ satisfy \eqref{eq:basictrace},
and $0\le \tr_\omega(\theta)\le C$.
\end{lemma}

\begin{proof}
Differentiating \eqref{eq:MA}, we obtain
\[
 \ddbar\dot\varphi=-\Ric(\omega)+\omega_0-\theta,
 \qquad \ddot\varphi=-R.
\]
Substituting into \eqref{eq:Udef} gives the first two identities in
\eqref{eq:potential}. Taking the trace of the first identity, we find
\[
 \Delta U=n-\tr_\omega(\theta)-(1-t)R,
\]
and hence $\square U=\tr_\omega(\theta)$. Since $R\ge-C$, the identity
for $U_t$ gives a uniform upper bound for $U$, while the minimum
principle applied to $\square U=\tr_\omega(\theta)\ge0$ gives a uniform
lower bound.

By the K\"ahler Bochner formula,
\begin{align*}
 \square|\partial U|^2
 &=-|\nabla^{2,0}U|^2-|\partial\bar\partial U|^2
       +2\operatorname{Re}\ip{\partial \tr_\omega(\theta)}{\partial U}\\
 &\le |\partial U|^2+|\partial \tr_\omega(\theta)|^2.
\end{align*}
For a fixed sufficiently large $A$, Proposition~\ref{prop:basic} shows
that $F=|\partial U|^2+A\tr_\omega(\theta)$ satisfies
\[
 \square F\le F+C.
\]
Applying the maximum principle to $e^{-t}(F+C)$ proves the gradient
bound. We note that all these estimates use only the lower bound for the
scalar curvature.
\end{proof}

For each base point $p=(x,t)$ and each earlier time $0\le s<t$, let
$\tau=t-s$ and define
\begin{equation}\label{eq:defect}
 D_p(\tau)=\inf_{d\zeta=0}\int_X
       |\dc U(s)-\zeta|^2_{g(s)/\tau}\,d\nu_{x,t;s}.
\end{equation}
Here the infimum is taken over smooth closed real one-forms on $X$, and
by Lemma~\ref{lem:projection} it is attained. More precisely, let
$\mathcal H_{p,\tau}$ denote the weighted $L^2$ completion of the space
of smooth real one-forms, and let $\mathcal Z_{p,\tau}$ be the closure
in it of the subspace of smooth closed forms:
\[
 \mathcal H_{p,\tau}=L^2\bigl(T^*X,g(s)/\tau;d\nu_{x,t;s}\bigr),
 \qquad
 \mathcal Z_{p,\tau}
 =\overline{\{\zeta\in C^\infty(T^*X):d\zeta=0\}}^{\mathcal H_{p,\tau}}.
\]
Then $D_p(\tau)$ is the square of the $L^2$ distance from $\dc U(s)$ to
the closed forms or, equivalently, the squared norm of its class in the
$L^2$ space of one-forms modulo closed forms:
\[
 D_p(\tau)
 =\operatorname{dist}_{\mathcal H_{p,\tau}}\!\bigl(\dc U(s),\mathcal Z_{p,\tau}\bigr)^2
 =\bigl\|[\dc U(s)]\bigr\|_{\mathcal H_{p,\tau}/\mathcal Z_{p,\tau}}^2.
\]

Adding a local pluriharmonic function to $U$ changes $\dc U$ by a
closed form and leaves $d\dc U=2\ddbar U$ unchanged, which explains the
role of closed forms in \eqref{eq:defect}.

\section{Integral bound for scalar curvature}\label{sec:scalar}

In this section, we relate the scalar curvature to the weighted integral
\eqref{eq:defect}. The estimate relies on a scalar Bochner calculation
and on the Hodge heat evolution of the minimizing closed form.

\begin{proposition}\label{prop:scalar}
For $t\ge1/2$ and $\lambda=1-t$,
\begin{equation}\label{eq:scalarreduction}
 \lambda R(x,t)\le C\left(1+\frac{D_{(x,t)}(\lambda)}{\lambda^2}\right).
\end{equation}
\end{proposition}

\begin{proof}
\emph{Step 1.}
Fix $t\ge1/2$ and $\lambda=1-t$. Rescale $[t-\lambda,t]$ to $[-1,0]$
by replacing $\omega(t+\lambda s)$ and the fixed reference form by
$\lambda^{-1}\omega(t+\lambda s)$ and $\lambda^{-1}\theta$, respectively.
Write these forms as $\omega(s)$ and $\theta$, denote the corresponding
real metric and scalar curvature by $g(s)$ and $R$, and set
\[
 \nu_s=\nu_{x,t;t+\lambda s},\qquad
 v(s)=\lambda^{-1}U(t+\lambda s),\qquad
 1-(t+\lambda s)=\lambda(1-s).
\]
Then \eqref{eq:potential} and the trace bound give
\begin{equation}\label{eq:scaledpotential}
 \begin{gathered}
 \square v=\tr_\omega(\theta),\qquad
 \ddbar v=\omega-(1-s)\Ric(\omega)-\theta,\\
 0\le\tr_\omega(\theta)\le C,\qquad |\theta|_\omega\le C.
 \end{gathered}
\end{equation}
We also have $R\ge-C\lambda$ and
\begin{equation}\label{eq:scaledtrace}
 \square\tr_\omega(\theta)
 \le-c|\partial\tr_\omega(\theta)|^2+C\lambda.
\end{equation}
Let $w$ solve
\[
 \square w=\tr_\omega(\theta),\qquad w(-1)=0.
\]
The maximum principle and \eqref{eq:duality} yield
\[
 \begin{gathered}
 0\le w(s)\le C(s+1)\le C,\qquad \square(w-v)=0,\\
 \square w^2=2w\tr_\omega(\theta)-2|\partial w|^2,\\
 \frac{d}{ds}\int_Xw^2\,d\nu_s
 =2\int_Xw\tr_\omega(\theta)\,d\nu_s
  -2\int_X|\partial w|^2\,d\nu_s.
 \end{gathered}
\]
Integrating from $-1$ to $0$, using $\nu_s\to\delta_x$ as $s\nearrow0$,
we obtain
\begin{equation}\label{eq:wenergy}
 \begin{aligned}
 \int_{-1}^0\!\int_X|\partial w|^2\,d\nu_s\,ds
 &=\int_{-1}^0\!\int_Xw\tr_\omega(\theta)\,d\nu_s\,ds
       -\tfrac12w(x,0)^2\\
 &\le C.
 \end{aligned}
\end{equation}

\smallskip
\Needspace{5\baselineskip}
\noindent\emph{Step 2.}
Apply Lemma~\ref{lem:projection} at time $-1$:
\[
 \dc v(-1)=\zeta(-1)+\beta,\qquad d\zeta(-1)=0,
\]
\[
 \begin{aligned}
 \|\beta\|_2^2
 &=\inf_{d\gamma=0}\int_X|\dc v(-1)-\gamma|^2\,d\nu_{-1}\\
 &=\lambda^{-2}\inf_{d\gamma=0}
   \int_X|\dc U(t-\lambda)-\lambda\gamma|_{g(-1)}^2\,d\nu_{-1}\\
 &=\lambda^{-2}D_{(x,t)}(\lambda).
 \end{aligned}
\]
Evolve $\zeta$ by the Hodge heat equation and define
\[
 \begin{gathered}
 \zeta(s)=\HH_{-1,s}\zeta(-1),\qquad d\zeta(s)=0,\\
 B=\ddbar(w-v),\qquad
 \eta=\bar\partial(w-v)-\sqrt{-1}\zeta^{0,1}.
 \end{gathered}
\]
By Lemma~\ref{lem:hodgeheat} and $\square(w-v)=0$,
\[
 \begin{gathered}
 \partial_s\eta=-\Delta_{\bar\partial,\omega(s)}\eta,\\
 \eta(-1)=-\bar\partial v(-1)-\sqrt{-1}\zeta^{0,1}(-1)
 =\sqrt{-1}\beta^{0,1},\qquad
 \|\eta(-1)\|_2^2=\|\beta\|_2^2.
 \end{gathered}
\]
Since $\zeta$ is real and closed,
\[
 \nabla_i\zeta_{\bar j}=\nabla_{\bar j}\zeta_i
 =\overline{\nabla_j\zeta_{\bar i}},
\]
and therefore
\[
 \begin{gathered}
 \nabla_i\eta_{\bar j}
 =(w-v)_{i\bar j}-\sqrt{-1}\nabla_i\zeta_{\bar j},\\
 \tfrac12\left(\nabla_i\eta_{\bar j}
       +\overline{\nabla_j\eta_{\bar i}}\right)
 =(w-v)_{i\bar j},\qquad
 |B|^2\le|\nabla^{1,0}\eta|^2.
 \end{gathered}
\]
Integrating \eqref{eq:complexheat} against $\nu_s$ now gives
\begin{equation}\label{eq:Bhessian}
 \begin{aligned}
 \int_{-1}^0\!\int_X|B|^2\,d\nu_s\,ds
 &\le\int_{-1}^0\!\int_X
       \bigl(|\nabla^{1,0}\eta|^2+|\nabla^{0,1}\eta|^2\bigr)
       \,d\nu_s\,ds\\
 &=\|\eta(-1)\|_2^2-|\eta(x,0)|^2\\
 &\le\|\beta\|_2^2=\lambda^{-2}D_{(x,t)}(\lambda).
 \end{aligned}
\end{equation}

\smallskip
\Needspace{5\baselineskip}
\noindent\emph{Step 3.}
Set $e=|\partial w|^2$ and $V=(1-s)^2R/2+e$. The scalar curvature
evolution and the Bochner formula give
\begin{align*}
 \square\left(\frac{(1-s)^2}{2}R\right)
 &=\frac{(1-s)^2}{2}|\Ric|^2-(1-s)R,\\
 \square e
 &=-|\nabla^{2,0}w|^2-|\partial\bar\partial w|^2
   +2\operatorname{Re}\ip{\partial\tr_\omega(\theta)}{\partial w}.
\end{align*}
In an $\omega$-unitary frame, regard $(1,1)$-forms as Hermitian matrices
and set
\[
 A=(1-s)\Ric,\qquad E=B-\theta,\qquad
 \ddbar w=B+\ddbar v=I+E-A.
\]
The curvature and complex Hessian terms satisfy
\begin{align*}
 \tfrac12|A|^2-\tr A-|I+E-A|^2
 &=-\tfrac12|A|^2+\ip{A}{I+2E}-|I+E|^2\\
 &=-\tfrac12|A-(I+2E)|^2
   +\tfrac12|I+2E|^2-|I+E|^2,\\
 \tfrac12|I+2E|^2-|I+E|^2
 &=\tfrac12(n+4\tr E+4|E|^2)-(n+2\tr E+|E|^2)\\
 &=|E|^2-\tfrac n2.
\end{align*}
Consequently,
\begin{equation}\label{eq:cancellation}
 \begin{split}
 \square V={}&-|\nabla^{2,0}w|^2
 -\tfrac12|(1-s)\Ric-(\omega+2B-2\theta)|^2\\
 &+|B-\theta|^2-\tfrac n2
   +2\operatorname{Re}\ip{\partial\tr_\omega(\theta)}{\partial w}.
 \end{split}
\end{equation}
Write $R\ge-C_R\lambda$, choose $A_0\ge2/c$ and $C_0\ge2C_R$, and set
\[
 W=V+A_0\tr_\omega(\theta)+C_0\lambda.
\]
For $-1\le s\le0$, we have
\[
 W-e\ge-2C_R\lambda+C_0\lambda\ge0,
\]
and Young's inequality gives
\[
 2\operatorname{Re}\ip{\partial\tr_\omega(\theta)}{\partial w}
 \le\frac{cA_0}{2}|\partial\tr_\omega(\theta)|^2
       +\frac{2}{cA_0}e.
\]
Combining \eqref{eq:scaledtrace} and \eqref{eq:cancellation}, we find
\begin{equation}\label{eq:Winequality}
 \begin{aligned}
 \square W
 &\le2|B|^2+2|\theta|^2
       -\frac{cA_0}{2}|\partial\tr_\omega(\theta)|^2
       +\frac{2}{cA_0}e+CA_0\lambda\\
 &\le2|B|^2+e+C
 \le W+2|B|^2+C.
 \end{aligned}
\end{equation}

\smallskip
\Needspace{5\baselineskip}
\noindent\emph{Step 4.}
For $-1\le a<0$, put $S(a)=\int_XR(a)\,d\nu_a$. By heat duality and
$\nu_a(X)=1$,
\[
 S'(a)=\int_X|\Ric|^2\,d\nu_a
 \ge\frac1n\int_XR^2\,d\nu_a
 \ge\frac1n\left(\int_XR\,d\nu_a\right)^2
 =\frac1nS(a)^2.
\]
If $S(a)>0$, then $S(s)\ge S(a)>0$ for $a\le s\le0$, and
\[
 \left(\frac1S\right)'\le-\frac1n,
 \qquad
 \frac1{S(a)}\ge\frac1{S(0)}+\frac{-a}{n}\ge\frac{-a}{n}.
\]
The resulting bound also holds when $S(a)\le0$. Thus
\[
 S(a)\le\frac{n}{-a}\quad(-1\le a<0),\qquad
 S(a)\le2n\quad(-1\le a\le-1/2).
\]
For $-1\le a\le-1/2$, the trace bound gives
\[
 \begin{aligned}
 \int_XW(a)\,d\nu_a
 &=\frac{(1-a)^2}{2}S(a)+\int_Xe(a)\,d\nu_a
       +A_0\int_X\tr_{\omega(a)}(\theta)\,d\nu_a+C_0\lambda\\
 &\le C+\int_Xe(a)\,d\nu_a.
 \end{aligned}
\]
Integrating and using \eqref{eq:wenergy}, we obtain
\[
 \int_{-1}^{-1/2}\!\int_XW(a)\,d\nu_a\,da
 \le C+\int_{-1}^0\!\int_X|\partial w|^2\,d\nu_s\,ds\le C.
\]
Moreover, \eqref{eq:Winequality} implies
\[
 \frac{d}{ds}\left(\exp(-s)\int_XW(s)\,d\nu_s\right)
 \le \exp(-s)\left(2\int_X|B|^2\,d\nu_s+C\right).
\]
Integrating from $a$ to $0$ gives
\[
 W(x,0)\le \exp(-a)\int_XW(a)\,d\nu_a
 +2\int_a^0\exp(-s)\!\int_X|B|^2\,d\nu_s\,ds
 +C\int_a^0\exp(-s)\,ds.
\]
Average over $a\in[-1,-1/2]$ and apply \eqref{eq:Bhessian}:
\[
 \begin{aligned}
 W(x,0)
 &\le2\exp(1)\int_{-1}^{-1/2}\!\int_XW(a)\,d\nu_a\,da
       +2\exp(1)\int_{-1}^0\!\int_X|B|^2\,d\nu_s\,ds+C\\
 &\le C\left(1+\lambda^{-2}D_{(x,t)}(\lambda)\right).
 \end{aligned}
\]
Finally, returning to the original scalar curvature,
\[
 \lambda R_{\mathrm{original}}(x,t)=R(x,0)
 \le2W(x,0)
 \le C\left(1+\frac{D_{(x,t)}(\lambda)}{\lambda^2}\right),
\]
which proves \eqref{eq:scalarreduction}.
\end{proof}

\section{A vanishing lemma}\label{sec:vanishing}

In general, the Hodge heat flow does not exactly preserve the condition
$\delta_{f(s)}\beta(s)=0$. The next calculation explains why the
weighted divergence condition imposed initially can nevertheless be used
at later times.

\begin{lemma} \label{lem:divergence}
Let $\beta$ be a solution of the Hodge heat equation on a real Ricci
flow, equipped with a conjugate heat measure based at time $0$,
and write $\tau=-s>0$.
If $\tau\square h=h-c$ for some constant $c$, then
\begin{equation}\label{eq:divergence}
 \square(\tau\delta_h\beta)
 =-2\tau\ip{\Ric_g+\nabla^2h-g/(2\tau)}{\nabla \beta}.
\end{equation}
Here $\beta$ need not be closed.
\end{lemma}

\begin{proof}
In real coordinates, $\delta \beta=-g^{ij}\nabla_i \beta_j$. Along the
Ricci flow,
\[
 \partial_s\Gamma^k_{ij}
 =-g^{ka}(\nabla_i\Ric_{ja}+\nabla_j\Ric_{ia}-\nabla_a\Ric_{ij}).
\]
The contraction of this expression in $i,j$ vanishes by the contracted
Bianchi identity. Since $\delta\Delta_d=\Delta_d\delta$ at each fixed
time, we obtain
\begin{equation}\label{eq:divunweighted}
 \square(\delta \beta)=-2\ip{\Ric_g}{\nabla \beta}.
\end{equation}
For the other term in $\delta_h\beta=\delta \beta+\ip{dh}{\beta}$, the
identities $\Delta^{\rm rough}dh=d\Delta h+\Ric_g^{\#}dh$ and
$\partial_s\beta=\Delta^{\rm rough}\beta-\Ric_g^{\#}\beta$ give
\begin{equation}\label{eq:divweighted}
 \square\ip{dh}{\beta}=\ip{d\square h}{\beta}-2\ip{\nabla^2h}{\nabla \beta}.
\end{equation}
Here again, the variation of the inverse metric cancels the Ricci terms.
Multiplying the sum of \eqref{eq:divunweighted} and
\eqref{eq:divweighted} by $\tau$, using $\tau'=-1$, and noting that
$-\delta \beta=\tr_g\nabla \beta$, we obtain
\[
 \square(\tau\delta_h\beta)
 =-2\tau\ip{\Ric_g+\nabla^2h-g/(2\tau)}{\nabla \beta}
      +\ip{\tau d\square h-dh}{\beta}.
\]
The last term vanishes by the equation assumed for $h$.
\end{proof}

\begin{lemma} \label{lem:vanish}
In the setting of Theorem~\ref{thm:compactness}, let $\beta_i$ be
solutions of the Hodge heat equation on $[-8,-1/2]$ satisfying
\[
 \|\beta_i(-8)\|_{L^2(\nu_{i,-8})}=1,
 \qquad \delta_{f_i(-8)}\beta_i(-8)=0.
\]
Suppose that the $\beta_i$ converge smoothly on compact subsets of the
regular spacetime for $-8<s<-1/2$, and that their limit at $s=-1$ is
closed. Then
\begin{equation}\label{eq:strongvanishing}
 \|\beta_i(-1)\|_{L^2(\nu_{i,-1})}\longrightarrow0.
\end{equation}
\end{lemma}

\begin{proof}
We first show that the limiting form is weighted co-closed.
The energy identity gives
\begin{equation}\label{eq:bi-energy}
 \sup_{[-8,-1]}\|\beta_i\|_2\le1,
 \qquad\int_{-8}^{-1}\!\int|\nabla \beta_i|^2d\nu_{i,s}\,ds\le\tfrac12.
\end{equation}
We take the potentials $h_i$ on a larger interval, such as
$[-16,-1/2]$. The initial divergence then satisfies
\[
 \|8\delta_{h_i}\beta_i(-8)\|_1
 =8\|\ip{d(h_i-f_i)}{\beta_i}(-8)\|_1
 \le8\|d(h_i-f_i)(-8)\|_2\longrightarrow0.
\]
Applying the scalar absolute-value inequality to \eqref{eq:divergence},
integrating by heat duality, and using the Cauchy-Schwarz inequality in
spacetime, we obtain
\begin{align*}
 \|\delta_{h_i}\beta_i(-1)\|_1
 &\le 8\|d(h_i-f_i)(-8)\|_2\\
 &\quad+16\left(\int_{-8}^{-1}\!\int
    |\Ric_{g_i}+\nabla^2h_i-g_i/(2\tau)|^2d\nu_{i,s}\,ds\right)^{1/2}
 \\[-2pt]
 &\hspace{26mm}\cdot
    \left(\int_{-8}^{-1}\!\int|\nabla \beta_i|^2d\nu_{i,s}\,ds\right)^{1/2}
 \longrightarrow0.
\end{align*}
Passing from $h_i$ back to $f_i$ at time $-1$ costs at most
$\|d(h_i-f_i)(-1)\|_2\|\beta_i(-1)\|_2=o(1)$. Thus
$\|\delta_{f_i}\beta_i(-1)\|_1\to0$, and by smooth local convergence,
$\delta_f \beta=0$ on the limiting regular slice. We emphasize that this
is real weighted co-closedness; no statement about a complex weighted
adjoint is inferred from it.

Next, we apply the weighted Bochner formula on the limiting slice.
By Lemma~\ref{lem:hodgeheat}, $F_i=|\beta_i|$ is a nonnegative heat
subsolution. In Lemma~\ref{lem:heat}, choose
\[
 a=-8,\qquad b=-1,\qquad p=2,\qquad q=8.
\]
The required condition holds because
\[
 \frac{|a|}{|b|}=8\ge\frac{q-1}{p-1}=7.
\]
Thus \eqref{eq:hyper} and the normalization at time $-8$ give
\begin{equation}\label{eq:L8}
 \|\beta_i(-1)\|_{L^8(\nu_{i,-1})}
 \le\|\beta_i(-8)\|_{L^2(\nu_{i,-8})}=1.
\end{equation}
Write $\mu=\nu_{-1}$ on the limiting slice. By local convergence and
exhaustion, $\beta\in L^8(\mu)$, and hence $\beta\in L^2(\mu)$.
Moreover, $d\beta=0$ by assumption, and $\delta_f\beta=0$ by the
preceding paragraph. The weighted Weitzenb\"ock formula and the soliton
equation then imply
\begin{equation}\label{eq:weightedbochner}
 0=(d\delta_f+\delta_fd)\beta
 =\nabla_f^*\nabla \beta+(\Ric_g+\nabla^2f)^{\#}\beta
 =\nabla_f^*\nabla \beta+\tfrac12\beta.
\end{equation}
Let $\chi$ be a smooth cutoff function with compact support in the
regular part. Multiplying by $\chi^2\beta$, integrating, and bounding
the cross term by Young's inequality, we obtain
\begin{equation}\label{eq:cutoffbochner}
 \tfrac12\int\chi^2|\nabla \beta|^2d\mu
 +\tfrac12\int\chi^2|\beta|^2d\mu
 \le2\int|d\chi|^2|\beta|^2d\mu.
\end{equation}
We now take $\chi=\rho_R\chi_\eps$, with the cutoffs of
Theorem~\ref{thm:compactness}. For fixed $R$, H\"older's inequality
gives
\[
 \int_{\operatorname{supp}\rho_R}|d\chi_\eps|^2|\beta|^2d\mu
 \le\|\beta\|_{L^8(\mu)}^2
       \|d\chi_\eps\|_{L^{8/3}(\operatorname{supp}\rho_R,\mu)}^2
 \longrightarrow0.
\]
Here the convergence in $L^{8/3}$ follows from the cutoff estimate in
$L^3(dV_g)$, the finiteness of the volume, and the local upper bound for
the heat density. The outer cutoff contributes at most
$CR^{-2}\|\beta\|_2^2$. Letting first $\eps\to0$ and then $R\to\infty$
in \eqref{eq:cutoffbochner}, we conclude that $\beta=0$. Note that all
integrations by parts took place on smooth compact subsets of the
regular locus.

Finally, we prove the strong convergence of the weighted norms; the
vanishing of the smooth limit alone would not imply
\eqref{eq:strongvanishing}. Choose compact subsets $K_j$ of the regular
part with $\mu(K_j)\to1$, and let $K_{j,i}$ be their images in the
approximating flows. By smooth convergence,
\[
 \nu_{i,-1}(K_{j,i})\to\mu(K_j),\qquad
 \int_{K_{j,i}}|\beta_i(-1)|^2d\nu_{i,-1}\to0.
\]
On the complement $E_{j,i}=X\setminus K_{j,i}$, H\"older's inequality
with conjugate exponents $4$ and $4/3$, followed by \eqref{eq:L8}, gives
\[
 \begin{aligned}
 \int_{E_{j,i}}|\beta_i(-1)|^2d\nu_{i,-1}
 &\le\left(\int_{E_{j,i}}|\beta_i(-1)|^8d\nu_{i,-1}\right)^{1/4}
       \nu_{i,-1}(E_{j,i})^{3/4}\\
 &\le\nu_{i,-1}(E_{j,i})^{3/4}.
 \end{aligned}
\]
Combining the estimates on $K_{j,i}$ and $E_{j,i}$, we obtain
\[
 \limsup_{i\to\infty}\|\beta_i(-1)\|_{L^2(\nu_{i,-1})}^2
 \le\bigl(1-\mu(K_j)\bigr)^{3/4}\longrightarrow0
 \quad\text{as }j\to\infty.
\]
This proves \eqref{eq:strongvanishing}.
\end{proof}

\section{Proof of Theorem 1.1}\label{sec:proof}

We now return to the original normalized flow and to the definition
\eqref{eq:defect}. Throughout this section, all comparisons are made
with the base point fixed.

\begin{lemma} \label{lem:coarse}
For every $0<\tau\le t/8$,
\begin{equation}\label{eq:coarse}
 D_p(\tau)\le C\tau,\qquad
 D_p(\tau)\le\tfrac14D_p(8\tau)+C\tau^2.
\end{equation}
\end{lemma}

\begin{proof}
Fix $p=(x,t)$. Since $U$ is real, \eqref{eq:norms} gives
\[
 (\dc U)^{0,1}=\sqrt{-1}\bar\partial U,\qquad
 |\dc U|_g^2=|(\dc U)^{0,1}|_\omega^2
 =|\bar\partial U|_\omega^2=|\partial U|_\omega^2.
\]
For any one-form $\beta$, we also have
\[
 (g/\tau)^{-1}=\tau g^{-1},\qquad
 |\beta|_{g/\tau}^2=\tau|\beta|_g^2.
\]
Taking $\zeta=0$ in the infimum \eqref{eq:defect}, we obtain
\[
 \begin{aligned}
 D_p(\tau)
 &\le\int_X|\dc U(t-\tau)|_{g(t-\tau)/\tau}^2\,d\nu_{x,t;t-\tau}\\
 &=\tau\int_X|\partial U(t-\tau)|_{\omega(t-\tau)}^2\,d\nu_{x,t;t-\tau}\\
 &\le C\tau\int_Xd\nu_{x,t;t-\tau}=C\tau,
 \end{aligned}
\]
where we used \eqref{eq:basicU} and $\nu_{x,t;t-\tau}(X)=1$.

For the second inequality, set
\[
 g_\tau(s)=\tau^{-1}g(t+\tau s),\qquad
 \nu_s=\nu_{x,t;t+\tau s}.
\]
Below, norms use $g_\tau(s)$ and $\nu_s$ at the indicated time, and $U$
denotes the pullback of the original potential. Rescale both the
K\"ahler form and the reference form by $\tau^{-1}$, writing them as
$\omega(s)$ and $\theta$; then $\tr_\omega(\theta)$ is unchanged and
$\square U=\tau\tr_\omega(\theta)$.

At time $-8$, choose the minimizing decomposition
$\dc U(-8)=\zeta+\beta(-8)$ from Lemma~\ref{lem:projection}, with
$d\zeta=0$. Since the conjugate heat measure is unchanged and
$g_\tau(-8)=8g(t-8\tau)/(8\tau)$, the squared norms are divided by eight:
\begin{equation}\label{eq:factor8}
 \|\beta(-8)\|_2^2=\tfrac18D_p(8\tau).
\end{equation}

The scalar heat propagator satisfies
\[
 \partial_a\Phi_{a,-1}=-\Phi_{a,-1}\Delta_{g_\tau(a)},\qquad
 \Phi_{-1,-1}=\operatorname{Id}.
\]
Hence
\[
 \begin{aligned}
 \frac{d}{da}\bigl(\Phi_{a,-1}U(a)\bigr)
 &=\Phi_{a,-1}(\partial_a-\Delta_{g_\tau(a)})U(a)\\
 &=\tau\Phi_{a,-1}\bigl(\tr_{\omega(a)}(\theta)\bigr),
 \end{aligned}
\]
and integration gives
\[
 U(-1)-\Phi_{-8,-1}U(-8)
 =\tau\int_{-8}^{-1}\Phi_{a,-1}
       \bigl(\tr_{\omega(a)}(\theta)\bigr)\,da.
\]
Applying $\dc$ and using \eqref{eq:commute}, we obtain
\begin{equation}\label{eq:duhamel}
 \begin{aligned}
 \dc U(-1)&=\HH_{-8,-1}\zeta+\HH_{-8,-1}\beta(-8)+\xi_\tau,\\
 \xi_\tau&=\tau\int_{-8}^{-1}\dc\Phi_{a,-1}
       \bigl(\tr_{\omega(a)}(\theta)\bigr)\,da.
 \end{aligned}
\end{equation}
By \eqref{eq:heatgradient} and the trace bound,
\begin{equation}\label{eq:duhamelerror}
 \|\xi_\tau\|_\infty
 \le C\tau\int_{-8}^{-1}(-1-a)^{-1/2}\,da
 =2\sqrt7\,C\tau\le C\tau.
\end{equation}
Since $\nu_{-1}(X)=1$, it follows that $\|\xi_\tau\|_2^2\le C\tau^2$.

Set $\zeta(s)=\HH_{-8,s}\zeta$ and
$\beta(s)=\HH_{-8,s}\beta(-8)$. Lemma~\ref{lem:hodgeheat} gives
\[
 d\zeta(-1)=0,\qquad
 \|\beta(-1)\|_2^2
 +2\int_{-8}^{-1}\!\int_X|\nabla\beta(s)|^2\,d\nu_s\,ds
 =\|\beta(-8)\|_2^2.
\]
Since $d\zeta(-1)=0$, the definition \eqref{eq:defect} gives
\[
 \begin{aligned}
 D_p(\tau)
 &\le\|\dc U(-1)-\zeta(-1)\|_2^2
 =\|\beta(-1)+\xi_\tau\|_2^2\\
 &\le2\|\beta(-1)\|_2^2+2\|\xi_\tau\|_2^2\\
 &\le2\|\beta(-8)\|_2^2+C\tau^2
 =\tfrac14D_p(8\tau)+C\tau^2.
 \end{aligned}
\]
\end{proof}

\begin{proposition} \label{prop:improvement}
For every $\eps>0$, there exist $K>8$, $\eta>0$, $C_\eps<\infty$ and
$\tau_*>0$ such that, with $p=(x,t)$ and $\lambda=1-t$,
\begin{equation}\label{eq:improvement}
 D_p(\tau)\le\eps D_p(8\tau)+C_\eps(\lambda+\tau)^2
\end{equation}
whenever $\tau\le \tau_*$, $K\tau\le t$, and
\begin{equation}\label{eq:goodscale}
 \Nash_p(\tau/K)-\Nash_p(K\tau)\le\eta.
\end{equation}
The constants are independent of the base point.
\end{proposition}

\begin{proof}
For $\eps\ge1/4$, Lemma~\ref{lem:coarse} gives
\[
 D_p(\tau)\le\tfrac14D_p(8\tau)+C\tau^2
 \le\eps D_p(8\tau)+C(\lambda+\tau)^2.
\]
Suppose the statement fails for some $0<\eps<1/4$. Choose
\[
 K_i\to\infty,\qquad \eta_i\to0,\qquad C_i\to\infty,
 \qquad\tau_{*,i}\to0,\qquad K_i\tau_{*,i}<1/4.
\]
There then exist $p_i=(x_i,t_i)$ and $\tau_i$ such that
\[
 \begin{gathered}
 0<\tau_i\le\tau_{*,i},\qquad K_i\tau_i\le t_i,\qquad
 \lambda_i=1-t_i,\\
 \Nash_{p_i}(\tau_i/K_i)-\Nash_{p_i}(K_i\tau_i)\le\eta_i,
 \end{gathered}
\]
but
\begin{equation}\label{eq:failure}
 D_{p_i}(\tau_i)>\eps D_{p_i}(8\tau_i)+C_i(\lambda_i+\tau_i)^2.
\end{equation}
By \eqref{eq:coarse},
\[
 0\le\lambda_i^2
 <\frac{D_{p_i}(\tau_i)}{C_i}
 \le\frac{C\tau_i}{C_i}\longrightarrow0,
 \qquad t_i\longrightarrow1.
\]
Set $A_i^2=\tfrac18D_{p_i}(8\tau_i)$. Combining
\eqref{eq:failure} and \eqref{eq:coarse} gives
\[
 \begin{aligned}
 \eps D_{p_i}(8\tau_i)+C_i(\lambda_i+\tau_i)^2
 &<\tfrac14D_{p_i}(8\tau_i)+C\tau_i^2,\\
 (C_i-C)(\lambda_i+\tau_i)^2
 &<8(\tfrac14-\eps)A_i^2.
 \end{aligned}
\]
Thus, for $i$ sufficiently large,
\begin{equation}\label{eq:smallratio}
 A_i>0,\qquad
 \frac{(\lambda_i+\tau_i)^2}{A_i^2}
 <\frac{8(1/4-\eps)}{C_i-C}\longrightarrow0.
\end{equation}

Rescale the flow by setting
\[
 \begin{gathered}
 g_i(s)=\tau_i^{-1}g(t_i+\tau_i s),\qquad
 \omega_i(s)=\tau_i^{-1}\omega(t_i+\tau_i s),\\
 U_i(s)=U(t_i+\tau_i s),\qquad
 \nu_{i,s}=\nu_{x_i,t_i;t_i+\tau_i s}.
 \end{gathered}
\]
All norms below use $g_i(s)$ and $\nu_{i,s}$ at the indicated time.
By Lemma~\ref{lem:projection} and \eqref{eq:factor8}, we may write
\[
 \begin{gathered}
 \dc U_i(-8)=\zeta_i+A_i\beta_i(-8),\qquad d\zeta_i=0,\\
 \|\dc U_i(-8)-\zeta_i\|_2^2
 =\tfrac18D_{p_i}(8\tau_i)=A_i^2,\\
 \|\beta_i(-8)\|_2=1,\qquad
 \delta_{f_i(-8)}\beta_i(-8)=0.
 \end{gathered}
\]
Let $\HH^{(i)}$ be the Hodge heat propagator of $g_i$, and set
\[
 \zeta_i(s)=\HH^{(i)}_{-8,s}\zeta_i,\qquad
 \beta_i(s)=\HH^{(i)}_{-8,s}\beta_i(-8).
\]
Lemma~\ref{lem:hodgeheat} gives, for $-8\le s\le-1/2$,
\[
 \begin{gathered}
 d\zeta_i(s)=0,\qquad
 \partial_s\beta_i=-\Delta_{d,g_i(s)}\beta_i,\\
 \|\beta_i(s)\|_2^2
 +2\int_{-8}^{s}\!\int_X|\nabla\beta_i|^2\,d\nu_{i,u}\,du=1.
 \end{gathered}
\]
Theorem~\ref{thm:compactness} applies. On regular spacetime cylinders
$Q'\Subset Q\Subset\RR_{(-8,-1/2)}$, use the smooth convergence maps
to identify the corresponding regions. Positivity of the limiting heat
density and interior parabolic estimates give
\[
 \begin{gathered}
 \frac{d\nu_{i,s}}{dV_{g_i(s)}}\ge c_Q>0,\qquad
 \int_Q|\beta_i|^2\,dV_{g_i(s)}\,ds
 \le c_Q^{-1}\int_{-8}^{-1/2}\!\int_X|\beta_i|^2\,d\nu_{i,s}\,ds
 \le C_Q,\\
 \|\beta_i\|_{C^k(Q')}
 \le C_{Q',Q,k}\|\beta_i\|_{L^2(Q)}\le C_{Q',Q,k}.
 \end{gathered}
\]
After taking a diagonal subsequence,
\[
 \beta_i\longrightarrow\beta
 \quad\text{in }C^\infty_{\mathrm{loc}}(\RR_{(-8,-1/2)}).
\]

By \eqref{eq:duhamel} and \eqref{eq:duhamelerror},
\[
 \dc U_i(-1)=\zeta_i(-1)+A_i\beta_i(-1)+\xi_i,
 \qquad \|\xi_i\|_\infty\le C\tau_i.
\]
Taking the exterior derivative gives
\begin{equation}\label{eq:closedlimit}
 d\beta_i(-1)=A_i^{-1}d\dc U_i(-1)-d(A_i^{-1}\xi_i).
\end{equation}
Here $\theta$ denotes the original fixed form in \eqref{eq:terminal}.
On every compact subset $K\Subset\RR_{-1}$, \eqref{eq:potential} and
the trace bound yield
\[
 \begin{gathered}
 \ddbar U_i(-1)
 =\tau_i\omega_i(-1)-(\lambda_i+\tau_i)\Ric(\omega_i(-1))-\theta,\\
 \|\Ric(\omega_i(-1))\|_{L^\infty(K)}\le C_K,\qquad
 |\theta|_{\omega_i(-1)}\le C\tau_i,\\
 \left\|A_i^{-1}d\dc U_i(-1)\right\|_{L^\infty(K)}
 \le C_K\frac{\lambda_i+\tau_i}{A_i}\longrightarrow0,\qquad
 \|A_i^{-1}\xi_i\|_\infty
 \le C\frac{\tau_i}{A_i}\longrightarrow0.
 \end{gathered}
\]
For a smooth compactly supported two-form $\psi$ on the regular slice,
with all forms pulled back by the convergence maps, integration by parts
therefore gives
\[
 \begin{aligned}
 \left|\int\ip{d(A_i^{-1}\xi_i)}{\psi}\,dV_{g_i(-1)}\right|
 &=\left|\int\ip{A_i^{-1}\xi_i}{\delta_{g_i(-1)}\psi}\,dV_{g_i(-1)}\right|\\
 &\le C_\psi\frac{\tau_i}{A_i}\longrightarrow0.
 \end{aligned}
\]
Together with \eqref{eq:closedlimit} and smooth local convergence, this
shows
\[
 d\beta_i(-1)\longrightarrow0\quad\text{in distributions},
 \qquad d\beta(-1)=0.
\]
Lemma~\ref{lem:vanish} now implies
\[
 \|\beta_i(-1)\|_2\longrightarrow0.
\]
Since $d\zeta_i(-1)=0$ and $\nu_{i,-1}(X)=1$, we conclude
\[
 \begin{aligned}
 \frac{D_{p_i}(\tau_i)}{A_i^2}
 &\le A_i^{-2}\|\dc U_i(-1)-\zeta_i(-1)\|_2^2\\
 &=\|\beta_i(-1)+A_i^{-1}\xi_i\|_2^2\\
 &\le2\|\beta_i(-1)\|_2^2
       +C\left(\frac{\tau_i}{A_i}\right)^2\longrightarrow0.
 \end{aligned}
\]
\Needspace{6\baselineskip}
On the other hand, \eqref{eq:failure} gives
\[
 \frac{D_{p_i}(\tau_i)}{A_i^2}
 >\eps\frac{D_{p_i}(8\tau_i)}{A_i^2}=8\eps>0,
\]
which is a contradiction.
\end{proof}

\begin{proof}[Proof of Theorem~\ref{thm:main}]
Fix $\eps<1/64$, together with the corresponding constants from
Proposition~\ref{prop:improvement}, and decrease $\tau_*$ if necessary
so that $K\tau_*<1/4$. Let $t>1/2$ with $\lambda=1-t$ small, fix
$p=(x,t)$, and set
\[
 \tau_j=8^j\lambda\quad(0\le j\le m),
 \qquad \tau_*/8<\tau_m\le \tau_*.
\]
All the required earlier times are then available. We call an index
$j<m$ good if \eqref{eq:goodscale} holds at $\tau_j$, and bad otherwise.

Let $N$ denote a uniform entropy bound supplied by
Lemma~\ref{lem:entropy}, so that $-N\le\Nash_p(\tau)\le0$.
The number of bad indices is uniformly bounded. Indeed, the
nonincreasing function $u\mapsto\Nash_p(e^u)$ has total decrease at most
$N$, and the entropy drop at $\tau_j$ is its decrease on the interval
\[
 [\log \tau_j-\log K,\ \log \tau_j+\log K].
\]
Since the centers of these intervals are separated by $\log8$, the
intervals have uniformly bounded overlap. If $M$ denotes their overlap
multiplicity, the sum of the drops is at most $MN$. As each bad index
contributes more than $\eta$, there are at most $B=\lceil MN/\eta\rceil$
bad indices.

At a good index we use \eqref{eq:improvement}, and at a bad index we use
\eqref{eq:coarse}. Since $\lambda\le \tau_j$, both give
\[
 D_p(\tau_j)\le a_jD_p(\tau_{j+1})+C\tau_j^2,
 \qquad a_j=\eps\ \text{or}\ \tfrac14.
\]
Every product of the first $j$ coefficients is at most $C_B\eps^j$,
where $C_B=(1/(4\eps))^B$, and iterating yields
\begin{equation}\label{eq:iteration}
 D_p(\lambda)
 \le C_B\eps^mD_p(\tau_m)
       +CC_B\lambda^2\sum_{j=0}^{m-1}(64\eps)^j.
\end{equation}
The series on the right is bounded. Set $a=-\log_8\eps>2$. By the
coarse bound and the fixed positive lower bound for $\tau_m$,
\[
 \eps^mD_p(\tau_m)
 \le C(\lambda/\tau_m)^a \tau_m
 =C\lambda^a \tau_m^{1-a}\le C\lambda^2.
\]
Thus $D_p(\lambda)\le C\lambda^2$, uniformly in $p$, and
Proposition~\ref{prop:scalar} gives $(1-t)R(x,t)\le C$.
\end{proof}


\bigskip
\bigskip

\begin{thebibliography}{99}
\bibitem{BamCompact}
R. H. Bamler,
\emph{Compactness theory of the space of super Ricci flows},
\href{https://arxiv.org/abs/2008.09298}{arXiv:2008.09298}.

\bibitem{BamHeat}
R. H. Bamler,
\emph{Entropy and heat kernel bounds on a Ricci flow background},
\href{https://arxiv.org/abs/2008.07093v3}{arXiv:2008.07093v3} (2021).

\bibitem{BamStructure}
R. H. Bamler,
\emph{Structure theory of non-collapsed limits of Ricci flows},
\href{https://arxiv.org/abs/2009.03243v2}{arXiv:2009.03243v2} (2021).

\bibitem{CCHZ}
C. Cifarelli, R. Conlon, M. Hallgren and J. Zhang,
\emph{Finite time singularities of the Ricci flow on compact
K\"ahler surfaces are of Type I},
\href{https://arxiv.org/abs/2609.16733}{arXiv:2609.16733} (2026).

\bibitem{CHM}
R. J. Conlon, M. Hallgren and Z. Ma,
\emph{Non-collapsed finite time singularities of the Ricci flow on
compact K\"ahler surfaces are of Type I},
\href{https://arxiv.org/abs/2502.19804}{arXiv:2502.19804} (2025).

\bibitem{HX}
C. Hacon and L. Xie,
\emph{On the K\"ahler MMP and the transcendental base-point-free theorem},
\href{https://arxiv.org/abs/2607.24986v1}{arXiv:2607.24986v1} (2026).

\bibitem{Hallgren}
M. Hallgren,
\emph{K\"ahler-Ricci tangent flows are infinitesimally algebraic},
\href{https://arxiv.org/abs/2312.06577v2}{arXiv:2312.06577v2} (2024).

\bibitem{HJ}
M. Hallgren and W. Jian,
\emph{Tangent flows of K\"ahler metric flows},
J. Reine Angew. Math. \textbf{805} (2023), 143-184.
\href{https://arxiv.org/abs/2202.06185}{arXiv:2202.06185}.

\bibitem{Hamilton}
R. S. Hamilton,
\href{https://doi.org/10.4310/jdg/1214436922}{\emph{Three-manifolds
with positive Ricci curvature}},
J. Differential Geom. \textbf{17} (1982), no.~2, 255-306.

\bibitem{JSFanoII}
W. Jian and J. Song,
\emph{Finite-Time Singularities of the K\"ahler-Ricci Flow on
Fano Bundles II},
\href{https://arxiv.org/abs/2609.20513}{arXiv:2609.20513} (2026).

\bibitem{JSTypeII}
W. Jian and J. Song,
\emph{Type I estimates for the K\"ahler-Ricci flow II},
forthcoming preprint.

\bibitem{JST}
W. Jian, J. Song and G. Tian,
\href{https://doi.org/10.1007/s00222-026-01449-x}{\emph{Finite time
singularities of the K\"ahler-Ricci flow}},
Invent. Math. (2026), published online 15 September 2026.
Preprint: \href{https://arxiv.org/abs/2310.07945v2}{arXiv:2310.07945v2}
(the numbering used here).

\bibitem{Perelman}
G. Perelman,
\emph{The entropy formula for the Ricci flow and its geometric applications},
\href{https://arxiv.org/abs/math/0211159}{arXiv:math/0211159} (2002).

\bibitem{SesumTian}
N. Sesum and G. Tian,
\href{https://doi.org/10.1017/S1474748008000133}{\emph{Bounding scalar
curvature and diameter along the K\"ahler Ricci flow (after Perelman)}},
J. Inst. Math. Jussieu \textbf{7} (2008), no.~3, 575-587.

\bibitem{STScalar}
J. Song and G. Tian,
\emph{Bounding scalar curvature for global solutions of the
K\"ahler-Ricci flow},
\href{https://arxiv.org/abs/1111.5681}{arXiv:1111.5681} (2011).

\bibitem{STMeasures}
J. Song and G. Tian,
\href{https://doi.org/10.1090/S0894-0347-2011-00717-0}{\emph{Canonical
measures and K\"ahler-Ricci flow}},
J. Amer. Math. Soc. \textbf{25} (2012), no.~2, 303-353.

\bibitem{ST}
J. Song and G. Tian,
\emph{The K\"ahler-Ricci flow through singularities},
Invent. Math. \textbf{207} (2017), 519-595.
\href{https://arxiv.org/abs/0909.4898v1}{arXiv:0909.4898v1}.

\bibitem{SW}
J. Song and B. Weinkove,
\href{https://doi.org/10.1215/00127094-1962881}{\emph{Contracting
exceptional divisors by the K\"ahler-Ricci flow}},
Duke Math. J. \textbf{162} (2013), no.~2, 367-415.

\bibitem{TianZhang}
G. Tian and Z. Zhang,
\href{https://doi.org/10.1007/s11401-005-0533-x}{\emph{On the K\"ahler-Ricci
flow on projective manifolds of general type}},
Chinese Ann. Math. Ser. B \textbf{27} (2006), no.~2, 179-192.

\bibitem{Tsuji}
H. Tsuji,
\href{https://doi.org/10.1007/BF01449219}{\emph{Existence and degeneration
of K\"ahler-Einstein metrics on minimal algebraic varieties of general type}},
Math. Ann. \textbf{281} (1988), no.~1, 123-133.


\bibitem{XZ}
T. Xu and Z. Zhang,
\emph{Collapsed Finite Time Singularities of the K\"ahler-Ricci Flow on Complex Surfaces are of Type I}, \href{https://arxiv.org/abs/2609.18834}{arXiv:2609.18834} (2026).
 


\bibitem{ZhangFinite}
Z. Zhang,
\emph{Scalar curvature behavior for finite time singularity of
K\"ahler-Ricci flow},
\href{https://arxiv.org/abs/0901.1474}{arXiv:0901.1474} (2009).
\end{thebibliography}
\end{document}